\documentclass[12pt, twosided]{amsart}

\usepackage{amssymb, hyperref, amsthm}
\usepackage[mathscr]{eucal}
\usepackage{xcolor}

\newtheorem{thm}{Theorem}[section]
\newtheorem{prop}[thm]{Proposition}
\newtheorem{cor}[thm]{Corollary}

\newtheorem*{thmA}{Theorem A}
\newtheorem*{thmB}{Theorem B}
\newtheorem*{thmC}{Theorem C}
\newtheorem*{thmD}{Theorem D}
\newtheorem*{thmE}{Theorem E}

\newtheorem*{propA}{Proposition A}
\newtheorem*{propB}{Proposition B}

\theoremstyle{definition}
\newtheorem{defn}[thm]{Definition}

\theoremstyle{remark}
\newtheorem{remark}[thm]{Remark}
\newtheorem{e.g.}[thm]{Example}

\usepackage{fancyhdr}

\newcommand{\lvC}{{\mathbb C}}

\newcommand{\lvN}{{\mathbb N}}
\newcommand{\lvP}{{\mathbb P}}
\newcommand{\lvQ}{{\mathbb Q}}
\newcommand{\lvR}{{\mathbb R}}
\newcommand{\lvZ}{{\mathbb Z}}

\newcommand{\dbM}{\mathfrak M}
\newcommand{\dbh}{\mathfrak h}

\newcommand{\sA}{\mathscr{A}}

\newcommand{\sE}{\mathscr{E}}

\newcommand{\sU}{\mathscr{U}}

\renewcommand{\Re}{{\mathfrak{R}}\mathfrak{e}\;}

\newcommand{\Pomt}{\lvP^1 \setminus \{0,1,\infty\}}

\newcommand{\Li}{\mathrm{Li}}

\newcommand{\alg}{\mathrm{alg}}

\newcommand{\st}{{\rm  st}}
\newcommand{\ST}{{\rm  ST}}
\newcommand{\pr}{{\rm  pr}}

\begin{document}

\title{On an algebraic version of the {\sc Knizhnik-Zamolodchikov} equation}
\author{$\;\;$Sheldon T Joyner
\newline
\tiny{The University of Western Ontario}}
\maketitle

\begin{abstract}
A difference equation analogue of the {\sc Knizhnik} - {\sc Zamolodchikov} equation is exhibited by developing a theory of the generating function $H(z)$ of {\sc Hurwitz} polyzeta functions to parallel that of the polylogarithms. By emulating the role of the KZ equation as a connection on a suitable bundle, a difference equation version of the notion of connection is developed for which $H(z)$ is a flat section.
Solving a family of difference equations satisfied by the {\sc Hurwitz} polyzetas leads to the normalized multiple {\sc Bernoulli} polynomials (NMBPs) as the counterpart to the {\sc Hurwitz} polyzeta functions, at tuples of negative integers. A generating function for these polynomials satisfies a similar difference equation to that of $H(z)$, but in contrast to the fact that said polynomials have rational coefficients, the algebraic independence of the {\sc Hurwitz} polyzeta functions is proven. The values of the NMBPs at $z=1$ provide a regularization of the multiple zeta values at tuples of negative integers, which is shown to agree with the regularization given in \cite{Akiyama:etal}. Various elementary properties of these values are proven.



\end{abstract}

\section*{Introduction}
The equation $\nabla G=0$ satisfied by the flat sections of the universal prounipotent bundle $\sU$ with connection $\nabla$ on $\Pomt$ is known as the formal {\sc Knizhnik-Zamolodchikov} (KZ) equation. A distinguished solution is given by the polylogarithm generating series $\Li(z)$, which is the {\sc Chen} series
\[
\Li(z, X_0,X_1) = \sum _{W}\int_{[0,z]}\omega_{W}W
\]
where the integrals are iterated, the sum is over all words $W$ in non-commuting variables $X_0$ and $X_1$, and if $\omega_0$ and $\omega_1$ denote $\frac{dz}{z}$ and $\frac{dz}{1-z}$ respectively, for any word $W=X_{i_1}\ldots X_{i_r}$ then
$\omega_{W} := \omega_{i_1}\ldots \omega_{i_r}$. 
The KZ equation arises in the study of the shuffle algebra of polyzeta values - see \cite{Cartier} for the details.

The goal of the present note is threefold: to  introduce an analogue of the KZ equation which emerges from purely algebraic considerations and is related to the stuffle algebra of polyzeta values,  to develop the notion of difference connection modelled on this equation to mimic the usual idea of (differential) connection on a manifold as in the case of the KZ equation, and finally to discuss normalized multiple {\sc Bernoulli} polynomials in this context. Here are some details:

The role of the polylogarithm functions in this setting is played by the {\sc Hurwitz} polyzeta functions
\[
\zeta(s_{1}, \ldots, s_{r}|z):=
\sum_{0 \leq n_{1} < \ldots < n_{r}}
\frac{1}{(z+n_{1})^{s_{1}}\ldots (z+n_{r})^{s_{r}}}\;\;\;,
\]
where the sums converge for $\Re (s_j+\ldots +s_{r})>r-j+1$ and determine meromorphic functions of $z$ with poles at non-positive integers. 
The {\sc Hurwitz} polyzeta functions may be seen to give rise to a homomorphism from a stuffle algebra of polynomials over $\lvC$ in infinitely many non-commuting variables $\{y_{j}\}_{j=1}^{\infty}$, to a stuffle algebra of meromorphic functions on $\lvC$ - i.e.   
\[
y_{j_{1}} \ldots y_{j_{r}} \mapsto \zeta(y_{j_{1}} \ldots y_{j_{r}}|z),
\] 
where $\zeta(y_{j_{1}} \ldots y_{j_{r}}|z)$ is some (sum of) {\sc Hurwitz} polyzeta function(s).

Introducing dual variables $\{Y_{j}\}_{j=1}^{\infty}$ to the $y_{j},$ the algebra  of non-commuting power series in the $Y_{j}$ with some ring $\Lambda$ of coefficients may be endowed with the structure of {\sc Hopf} algebra. When $\Lambda$ comprises a class of meromorphic functions of which the {\sc Hurwitz} polyzetas form a subclass, a distinguished element of this algebra ${\sA}$ is the generating series 
\[
H(z):=\sum_{y_{j_{1}} \ldots y_{j_{r}}}\zeta(y_{j_{1}}\ldots y_{j_{r}}|z)Y_{j_{1}} \ldots Y_{j_{r}},
\]
where the sum is taken over all words in the $y_{j}.$ From formal considerations (cf. \cite{RacinetThese}), it is evident that $H(z)$ is group-like - i.e. with respect to the comultiplication $\Delta$ in ${\sA}$, $$\Delta H(z) = H(z) \otimes H(z).$$

Moreover, making use of a family of elementary difference equations satisfied by the {\sc Hurwitz} polyzeta functions, we prove the
\begin{thmA}
\label{1}
If $D_{-}$ denotes the difference operator 
$D_{-}F(z) = F(z) - F(z-1),$ 
then
\begin{equation}
\label{DeQ0}
D_{-}H(z) =  -\sum_{k=1}^{\infty}\frac{Y_{k}}{(z-1)^{k}}H(z).
\end{equation}
\end{thmA}

We go on to define an algebraic analogue of the notion of connection, for which (\ref{DeQ0}) determines some kind of universal algebraic connection $\nabla^{\alg}$ with respect to which $H(z)$ forms a flat section. Emulating the unipotence property of  connections in this context, and defining an $\dbM_{0}$-bundle over the complex sphere with finitely many points removed to be any bundle with the space of meromorphic functions on the sphere with poles exactly at the non-positive integers as subbundle, the versal property is the content of the
\begin{thmB}
Given any $\dbM_{0}$-bundle ${\sE}$ with unipotent difference connection $\nabla^{alg}$ on $X$ along with a point $b\in X$ and any $v \in {\sE}_{b},$  there exists a mapping
\[
\psi_{v}: ({\sU}, \nabla^{A}) \rightarrow ({\sE}, \nabla^{alg})
\]
which is compatible with the difference connections and 
has 
$(\psi_{v})_{b} (1) = v.$
\end{thmB}

In the final section of the paper, a generating function for normalized multiple {\sc Bernoulli} polynomials is defined, and shown to satisfy a similar difference equation to (\ref{DeQ0}). Precisely, the above-mentioned difference equations may be considered for negative arguments. Using known solutions for the equations of first level given by the usual {\sc Bernoulli} polynomials, these difference equations may be solved explicitly. The polynomials $\zeta(-n_1, \ldots, -n_r|z)$ which arise are counterparts at negative integer parameters $(-n_1, \ldots, -n_r)$ to the {\sc Hurwitz} polyzeta functions defined above. Denoting by $\{Y_{j}\}_{j\leq 0}$ a countable set of non-commuting formal variables, one then sets
\[
H_B(z):=\sum\zeta(-n_1, \ldots, -n_r|z)Y_{-n_1}\ldots Y_{-n_r}
\]
where the sum is taken over all tuples $(-n_1, \ldots, -n_r)$ of non-positive integers. As before it is possible to prove
\begin{thmC}
\[
D_{-}H_B(z) =  -\sum_{k=0}^{\infty}{Y_{-k}}{(z-1)^{k}}H_B(z).
\]
\end{thmC}

It is interesting that while these normalized {\sc Bernoulli} polynomials have rational coefficients, the {\sc Hurwitz} polyzeta functions are transcendental. Indeed, we prove:
\begin{thmD}
The {\sc Hurwitz} polyzeta functions are algebraically independent over $\lvC$ - i.e. for any $N \geq 1$, if $P(T_1, \ldots, T_N) \in \lvC[T_1, \ldots, T_N]$ is a polynomial which vanishes (uniformly in $z$) at some $N$-tuple of any {\sc Hurwitz} polyzeta functions of the variable $z$, then $P$ is identically zero.
\end{thmD}

Evaluating the {\sc Hurwitz-Bernoulli} polyzeta functions at $z=1$ gives the polyzeta values (also known as multiple zeta values in the literature) in the case of the {\sc Hurwitz} polyzetas, and a set of possible regularizations of the polyzeta values at non-positive integers in the case of the normalized multiple {\sc Bernoulli} polynomials. It is known that many different regularizations are possible - cf. \cite{LiGuo} - and in the case of our difference equation method, it is clear that the values are not uniquely determined since the constant term of any polynomial vanishes under the action of $D_{-}$! However, it is possible to show that these regularizations coincide with those given by a different method in \cite{Akiyama:etal}. More precisely, by determining a recursive formula for the normalized {\sc Bernoulli} polynomials, one can prove the
\begin{thmE}
\[
\zeta(-k_1, \ldots, -k_r|1)
=
\lim_{s_1 \rightarrow -k_1}
\cdots
\lim_{s_r\rightarrow -k_r}
\zeta(s_1, \ldots, s_r),
\]
where $\zeta(s_1, \ldots, s_r)$ denotes the meromorphic continuation of the series
\[
\sum_{0<n_1\ldots<n_r}\frac{1}{n_1^{s_{1}}\ldots n_r^{s_r}}.
\]
\end{thmE}

Once this is proven, we establish various properties of the regularized polyzeta values using the explicit normalized multiple {\sc Bernoulli} polynomials, for example
\begin{propA}
If $n$ and $k$ are integers for which $n\geq 1$ and $0\leq k \leq 2n+1,$ then
\[
\zeta(0)\zeta(-2n-1) = \zeta(-2n+k-1,-k|1)
\]
\end{propA}
and
\begin{propB}
If $n \geq 1$ is any integer, then
\[
\zeta(-n,-n-1|1) = \zeta(-n-1,-n|1).
\]
\end{propB}

\section{The {\sc Hurwitz} polyzeta functions and the formal KZ difference equation}

With notation as in the introduction, on the space $\Lambda<<y_1, y_2, \ldots >>$ of formal power series in the non-commuting variables $\{y_j\}_{j=1}^{\infty}$ as in \cite{Cartier} define the stuffle product by induction via
\[
y_{k}w * y_{k'}w'
=
y_{k}(w*y_{k'}w') + y_{k'}(y_{k}w *w') + y_{k+k'}(w*w'),
\]
where 1 denotes the empty word, linearly extended to arbitrary formal power series. Those words in the $y_j$ which do not end in $y_1$ may be mapped to {\sc Hurwitz} polyzeta functions via the homomorphism described above. For words which end in $y_1,$ some regularization is required. This is achieved by setting
\[
\zeta(1|z):=-\frac{\Gamma '(z)}{\Gamma(z)}
\]
where $\Gamma(z)$ denotes the usual interpolation of the factorial function; then building up the remaining functions for which some regularization is required, using the stuffle product itself. The choice of this regularization is motivated by the classical expression:
\[
\zeta(r|z) = \frac{1}{r-1} - \frac{\Gamma'(z)}{\Gamma(z)} + O(r-1).\]
i.e.
\begin{equation}
\label{Zeta1}
\lim_{\varepsilon \rightarrow 0} \left(\zeta(1+\varepsilon;z) - \frac{1}{\varepsilon} \right)
= 
-\frac{\Gamma'(z)}{\Gamma(z)}
\end{equation}

This is an apt choice for our purposes since $\zeta(1|z)$ as defined here satisfies a functional equation of the form of similar equations satisfied by the other functions in the family of {\sc Hurwitz} polyzeta functions:
Directly from the definitions one may deduce:
\begin{equation}
\label{fun}
\zeta(k_{1}, \ldots, k_{r}|z+1) -\zeta(k_1,\ldots, k_r|z)
=  -\frac{1}{z^{k_{1}}}\zeta(k_{2}, \ldots, k_{r}| z+1),
\end{equation}
where if $r= 1,$
\begin{equation}
\label{clincher}
\zeta(k|z+1) -\zeta(k|z)= - \frac{1}{z^{k}}\;\;
{\mbox{for $k>1.$}}
\end{equation}
An elementary calculation using the functional equation $\Gamma(z+1) =z\Gamma(z)$ shows that also
\[
\zeta(1|z+1)-\zeta(1|z) = -\frac{1}{z}.
\]

\begin{remark} The regularization procedure discussed above admits of an interesting formulation involving concrete limits in the style of (\ref{Zeta1}):
To see this, we introduce non-commuting variables $y_{t}$ for real $t \geq 
1$ and  define a stuffle product on the non-commuting polynomial 
algebra $\Lambda <\{ y_{t}\}_{t \geq 1}>$ as before:
\[
y_{t_{1}} * y_{t_{2}} = y_{t_{1}}y_{t_{2}}+y_{t_{2}}y_{t_{1}} + 
y_{t_{1}+t_{2}} 
\] 
for any $t_{1}; t_{2}\geq 1.$ 
We denote this algebra by $\dbh_{\st, \lvR}.$
Also, write 
\[
\lim_{\varepsilon \rightarrow 0^{+}}
y_{k +\varepsilon} = y_{k}
\] 
for any 
$k \in \lvN\backslash\{0\}.$ 
Such limits clearly commute with the stuffle product. 

Now consider the {\sc Hurwitz} polyzeta functions 
$\zeta(s_{1}, \ldots, s_{r}|z)$ where the $s_{j} \in\lvR$ have $s_{j} \geq 1 $ for $j  = 1, 2, \ldots, r-1$ and $s_{r}>1.$ Such functions satisfy the stuffle relations and form a $\lvQ$-algebra  which is isomorphic to a subalgebra, say $\dbh_{\st, \lvR}^{0}$, of $\dbh_{\st, \lvR},$ when we take $\Lambda = \lvQ.$ To extend the correspondence to the entire $\dbh_{\st, \lvR},$ as before we employ the regularization for $\zeta(1|z)$ given above. Then (\ref{Zeta1}) is the analogue for the regularized {\sc Hurwitz} polyzeta functions of 
\[
\lim_{\varepsilon \rightarrow 0^{+}}y_{1+ \varepsilon} = y_{1}
\]
and evidently commutes with stuffle product among these polyzeta functions. In this way, any regularized {\sc Hurwitz} polyzeta function (at a tuple of integers) may be expressed as a limit which equals the stuffle product expression 
that can be built up as before to give the regularized value. The point is that regularization assigns meaning to a certain symbol, (which amounts to the assignment of a function of $z$ in the instances of the {\sc Hurwitz} polyzeta algebras), and here, via certain well-defined limits an alternative description of the regularization is supplied.
An example will clarify these ideas:
\begin{e.g.}
Here the limit regularization of $\zeta(k,1|z)$ for any integer $k>1$ will be elucidated:
Corresponding to the stuffle product
\[
y_{k}*y_{1+\varepsilon} = y_{k}y_{1+\varepsilon} + y_{1+\varepsilon}y_{k} + y_{k+1+\varepsilon}\]
we have the equality
\begin{align}
\label{5}
&
\left( \sum_{n=0}^{\infty} \frac{1}{(z+n)^{k}}\right) \left( 
\sum_{n=0}^{\infty}\frac{1}{(z+n)^{1+\varepsilon}} \right) 
\\
&= \sum_{0 \leq n_{1}<n_{2}}\frac{1}{(z+n_{1})^{k}(z+n_{2})^{1+\varepsilon}}
+\sum_{0 \leq n_{1}<n_{2}}\frac{1}{(z+n_{1})^{1+\varepsilon}(z+n_{2})^{k}}
+\sum_{n=0}^{\infty}\frac{1}{(z+n)^{k+1+\varepsilon}}
\notag
\end{align}
for $\varepsilon >0.$

Then 
\begin{align*}
&
\left( \sum_{n=0}^{\infty} \frac{1}{(z+n)^{k}}\right) 
\left( 
\sum_{n=0}^{\infty}\frac{1}{(z+n)^{1+\varepsilon}} \right) 
-
\left( 
\sum_{n=0}^{\infty} \frac{1}{(z+n)^{k}}\right) 
\left( 
\frac{1}{\varepsilon} 
\right)
\\
&=
\left( 
\sum_{n=0}^{\infty}\frac{1}{(z+n)^{k}} 
\right) 
\left( 
\sum_{n=0}^{\infty} \frac{1}{(z+n)^{1+\varepsilon}} - \frac{1}{\varepsilon}
\right) 
\end{align*} 
and in the limit as $\varepsilon \rightarrow 0,$ this expression approaches 
\[
\left( 
\sum_{n=0}^{\infty} \frac{1}{(z+n)^{k}}
\right) 
\left[-\frac{\Gamma(z)}{\Gamma(z)}\right] = \zeta(k|z) \zeta(1|z).
\]
Solving (\ref{5}) for the sum we are attempting to regularize, namely
\[
\sum_{0 \leq n_{1}<n_{2}}\frac{1}{(z+n_{1})^{k}(z+n_{2})^{1+\varepsilon}}
\]
and subtracting 
\[
\sum_{n=0}^{\infty} \frac{1}{(z+n)^{k}}
\cdot
\frac{1}{\varepsilon} 
\] from both sides of the result before finally taking the limit, we find that 
\[
\lim_{\varepsilon \rightarrow 0} \left( \zeta(k, 1+ \varepsilon|z) - \frac{\zeta(k|z)}{\varepsilon} 
\right)
=\zeta(k|z)\zeta(1|z) - \zeta(1,k|z) - \zeta(k+1|z),
\]
which is the regularized value of $\zeta(k,1|z).$
\end{e.g.}

As a sampling of similar computations we mention that
for any integers $k_{1}, \ldots, k_{r}$ where $k_{r} >1$,
\begin{align*}
&\lim_{\varepsilon \rightarrow 0}\left(
\zeta(k_{1}, \ldots, k_{r}, 1+\varepsilon|z) - \frac{\zeta(k_{1}, \ldots, k_{r}|z)}{\varepsilon} \right)
=
\zeta(k_{1}, \ldots, k_{r}|z)\zeta(1|z) \\
&
- \sum_{j=1}^{r} \zeta(k_{1}, \ldots, k_{j-1}, 1, k_{j}, \ldots, k_{r}|z) 
- \sum_{j=1}^{r}\zeta(k_{1}, \ldots, k_{j-1}, 
k_{j}+1, k_{j+1}, \ldots, k_{r}|z).
\end{align*}
Also, in regularizing $\zeta(2,1,1|z)$ we obtain
\begin{eqnarray*}
 & &
\lim_{\varepsilon \rightarrow 0} \lim_{\delta \rightarrow 0} \left( 
\zeta(2,1+\varepsilon, 1+\delta|z)
+
\zeta(2,1+\delta, 1+\varepsilon|z)
- \frac{\zeta(2|z)}{\varepsilon}
\zeta(1+\delta|z)\right.
\\
 & &
\left.
- \frac{\zeta(2|z)}{\delta}\left(
{\zeta({1+\varepsilon}|z)} - \frac{1}{\varepsilon} \right)
+ \frac{\zeta(3|z)}{\delta}
+ \frac{\zeta(3|z)}{\varepsilon}
+ \frac{\zeta(1,2|z)}{\delta}
+ \frac{\zeta(1,2|z)}{\varepsilon}
\right)
\\
&=& \zeta(1|z)^{2}\zeta(2|z) +\zeta(1|z)(-2\zeta(3|z)-2\zeta(1,2|z))+\zeta(4|z)+2\zeta(1,3|z)
+2\zeta(1,1,2|z).
\end{eqnarray*}
\end{remark}

\subsection{The algebraic KZ equation}
Introducing variables $\{Y_{j}\}_{j=1}^{\infty}$ dual to the $y_{j},$ we may form the graded {\sc Hopf} algebra dual to the algebra $\Lambda<y_{j}>_{j=1}^{\infty}$ of polynomials in the non-commuting $y_{j}$. 
In this case, the comultiplication is given by
\[
\Delta_{\ST}(Y_{k}) =
1 \otimes Y_{k}
+
\sum_{j=0}^{k-1}Y_{j} \otimes Y_{k-j}
+
Y_{k} \otimes 1.
\]
This is designed in such a way that in each term on the right hand side, the stuffle product of the two constituents of the tensor product has the left hand side as a term. As a consequence of this it is not hard to show by a formal argument that the generating function of {\sc Hurwitz} polyzeta functions
\[
H(z):= \sum_{y_{k_{1}} \ldots y_{k_{r}} \in Y^{*}}\zeta(y_{k_{1}} \ldots y_{ k_{r}}|z)Y_{k_{1}} \ldots Y_{k_{r}},
\]
(where the coefficient of the empty word is set to be $1$), is group-like - i.e. satisfies 
\[
\Delta_{ST}H(z) = H(z) \otimes H(z).
\]

The analogue for $H(z)$ of the differential equation satisfied by $Li(z)$ is, as disclosed above, of an algebraic nature:
The equations ({\ref{fun}}) can be brought together in a universal functional 
equation for {\sc Hurwitz} polyzeta functions, just as the formal KZ equation is a compilation of differential equations satisfied by the various polylogarithm functions. To be precise, we have the 
{\thm{
\label{t:AKZ}
If $D_{-}$ denotes the difference operator 
$D_{-}F(z) = F(z) - F(z-1),$ 
then
\begin{equation}
\label{DeQ}
D_{-}H(z) =  -\sum_{k=1}^{\infty}\frac{Y_{k}}{(z-1)^{k}}H(z).
\end{equation}
}}
{\bf Proof:}
From (\ref{fun}), it is clear that for fixed $k \geq 1,$ we have
\begin{eqnarray*}
\sum_{w \in Y^{*}}\zeta(y_{k}w|z)Y_{k}W
&=&
\sum_{w \in Y^{*}}\zeta(y_{k}w|z-1)Y_{k}W
-
\frac{Y_{k}}{(z-1)^{k}}\sum_{w \in Y^{*}}\zeta(w|z)
\\
&=&
\sum_{w \in Y^{*}}\zeta(y_{k}w|z-1)Y_{k}W
-
\frac{Y_{k}}{(z-1)^{k}}H(z)
\end{eqnarray*}
where $W$ is the word in the variables $Y_{j}$ which is dual to $w$.

But then certainly
\[
H(z) = H(z-1) - \sum_{k=1}^{\infty}\frac{Y_{k}}{(z-1)^{k}}H(z)
\]
and the statement of the theorem is immediate.
$\hfill \Box$

{\section{Difference connections}}

The {\sc Knizhnik-Zamolodchikov} equation characterizes flat sections 
of the universal unipotent bundle with connection on $\Pomt$. Although much less rich, an algebraic analogue can be ideated around (\ref{DeQ}). To this end, use
$
\dbM_{k}
$
to denote the family of meromorphic functions of a single complex variable which have poles at all integers less than or equal to some integer $k$. Then $\Gamma(z) \in \dbM_{0}$ as is $\zeta(s_{1}, \ldots, s_{r};z)$ (for suitable tuples of complex $s_{j}$) when viewed as a function of $z.$ Notice that  if $F(z) \in \dbM_{k},$ then $F(z-1) \in \dbM_{k+1}$, so the operator $D^{-}:F(z) \mapsto F(z-1)$ sends $\dbM_{k}$ onto $\dbM_{k+1}.$  Now let $\dbM_{\infty}$ denote the union of the $\dbM_{k}$ for all integer $k$, a space which is invariant under $D^{-}.$

We have $D^{-} = I-D_{-}$ and (\ref{DeQ}) is the same as
\[
\left(D^{-}-I-\sum_{k=1}^{\infty}\frac{Y_{k}}{(z-1)^{k}}\right) H(z)=0.
\]
The operator 
$D^{-}-I-\sum_{k=1}^{\infty}\frac{Y_{k}}{(z-1)^{k}}$ should play the role of an algebraic version of a connection.

An $\dbM_{\infty}$-bundle ${\sE}$ on $X:=\lvP^{1}_{\lvC} \backslash\{a_{1}, \ldots, a_{N}\}$ where the $a_{j} \in \lvZ \cup \{\infty\}$ is a bundle on $X$ having $\dbM_{k}$ as a subbundle for each integer $k$, where by an abuse of notation $\dbM_{k}=:\dbM_{k}(X)$ comprises those functions on $X$ which when viewed as functions on $\lvC$ using the usual coordinates are meromorphic with poles at all integers less than or equal to $k$. (We also require a pole at $\infty$ if the latter is a point of $X$.) An $\dbM_{k}$-bundle on $X$ is a bundle merely required to have $\dbM_{k}$ itself as subbundle.

Notice that $\dbM_{\infty}=:\dbM_{\infty}(X)$ may itself be regarded as an
$\dbM_{\infty}$-bundle on such an $X.$ This bundle is the analogue of the space of differential 1-forms on $X.$

With this notation, we make the 
\begin{defn}
Given an $\dbM_{0}$-bundle ${\sE}$ on $X$, a difference connection on  the bundle is a $\lvC$-linear mapping
\[
\nabla^{alg}:{\sE} \rightarrow {\dbM}_{\infty} \otimes_{ \dbM_{0}} {\sE}
\]
for which for any $f \in \dbM_{\infty}$ and any section $s$ of ${\sE},$
\begin{equation}
\label{Leib}
\nabla^{alg}(fs) =D^{-}f \cdot (-D_{-}s) + f (\nabla^{alg} s).
\end{equation}
\end{defn}
The condition (\ref{Leib}) should be regarded as an algebraic analogue of the usual {\sc Leibnitz} rule. 

In what follows, any bundle will be assumed to be an $\dbM_{0}$-bundle.

We require the following notion:
\begin{defn}
A bundle with difference connection 
$({\sE}, \nabla^{alg}) $ is called unipotent when 
there is some $r$ and some sequence of upper triangular matrices $\{N_{k}\}_{k\geq 1}$ such that
\[
({\sE}, \nabla^{alg})
\simeq
(\dbM_{0}^{r}, D^{-}-I-\sum_{k=1}^{\infty}\frac{N_{k}}{(z-1)^{k}}).
\]
\end{defn}

We proceed to construct an object which satisfies a versal property with respect to the bundles with unipotent difference connection.

Consider the algebra $\lvC<Y>$ of polynomials in the countable set of non-commuting variables $Y:=\{Y_{j}\}_{j=1}^{\infty}.$ Let $J$ denote the augmentation ideal $J = (Y_{1}, Y_{2}, \ldots).$ Then write
\[
U_{n}:= \lvC<Y> / J^{n+1}.
\]
This algebra comprises those  polynomials in which the words that appear have length at most $n$. Now form ${\sU}_{n}:=U_{n} \otimes \dbM_{0}$ and 
${\sU} = \lim_{\leftarrow}U_{n}.$ Observe that we could write ${\sU} = \dbM_{0}<<Y>>$ since effectively we are considering the formal power series algebra in the non-commuting variables $Y_{k}$ with coefficients taken from $\dbM_{0}.$ Also, notice that $H(z) \in {\sU}$ when we take $X$ to be $\lvC.$ 

On ${\sU}$ a difference connection $\nabla^{A}$ is determined by defining for each $n$, 
\begin{eqnarray*}
&&\nabla^{A}_{n}\left( 
\sum_{|w| \leq n}f_{w}(z)[w]
\right)
\\
&:=&
\sum_{|w|\leq n} f_{w}(z-1)[w] - 
\sum_{|w|\leq n}f_{w}(z)[w] - 
\pr_{n} \sum_{|w|\leq n}f_{w}(z)
\sum_{k=1}^{\infty}\frac{1}{(z-1)^{k}}[Y_{k}w]
\end{eqnarray*}
where $[w]$ denotes the class of the word $w$, and $\pr_{n}$ indicates  projection to words of length less than or equal to $n$.

\addtocounter{thm}{1}
\begin{thm}
\label{L}
Given any $\dbM_{0}$-bundle ${\sE}$ with unipotent difference connection $\nabla^{alg}$ on $X$ along with a point $b\in X$ and any $v \in {\sE}_{b},$  there exists a mapping
\[
\psi_{v}: ({\sU}, \nabla^{A}) \rightarrow ({\sE}, \nabla^{alg})
\]
which is compatible with the difference connections and 
has 
$(\psi_{v})_{b} (1) = v.$
\end{thm}
{\bf Proof:}
Corresponding to a word $w = Y_{i_{1}} \ldots Y_{i_{r}}$, set $N_{w}:= N_{i_{1}} \ldots N_{i_{r}}.$ Then define
\[
\psi: ({\sU}, \nabla^{A}) \rightarrow ({\sE}, \nabla^{alg})
\]
by taking
\[
\psi \left(
\sum_{w} f_{w}(z)[w] \right)
=
\sum_{w}f_{w}(z)N_{w}(z)\cdot v.
\]

Notice that $\psi_{b} (1) = 1 \cdot v$.

Now $N_{k}N_{w} = N_{Y_{k}w}.$ Because of this, one sees readily that $\psi \circ \nabla^{A} = \psi \circ \nabla^{alg}.\hfill \Box $

As in the topological case,
one can define a unipotent fundamental group $\pi_{1}^{DR,diff}(X,b)$ as 
the tensor compatible automorphisms of the fiber functor with respect to the category of unipotent {\em difference connections}. This group acts on sections of any bundle with difference connection, and by definition, the action commutes with the $\psi_{v}$ of the Theorem. 
This gives the ``parallel transport'' action on the bundle, but unlike the iterated integral situation, this action admits no description intrinsic to the flat section of the bundle (with respect to the {\em difference} connection). Moreover, while the parallel transport provides a means of showing that the analogous mapping to $\psi_{v}$ in the topological context is characterized by its action on the element $1$ of the fiber above $b$, via the linear independence of the iterated integral (polylogarithm) functions, even though the {\sc Hurwitz} polyzeta functions are linearly independent over $\lvC$, this linear independence does not interact suitably with the action of $\pi_{1}^{DR,diff}(X,b)$ to facilitate similar conclusions.

By construction, $H(z)$ is a flat section of 
$({\sU}, \nabla^{A})$.

\section{The normalized multiple {\sc Bernoulli} 
polynomials}
The classical {\sc Hurwitz} zeta function defined for $\Re s>1$ by
\[
\zeta(s,z)=\sum_{n=0}^{\infty}\frac{1}{(z+n)^{s}}
\]
may be analytically continued in the complex variable $s$: For example modifying {\sc Riemann's} contour integral approach to his zeta function $\zeta(s)$, it is easy to show that
\[
\zeta(s,z)= \frac{\Gamma(1-s)}{2 \pi i}\int_{C}\frac{e^{(1-z)w}}{e^w -1}{(-w)^{s}}\frac{dw}{w},
\]
where $C$ is the {\sc Hankel} contour in $\lvC$ (i.e.  a loop about 0 based at infinity, enclosing the positive real axis). This integral expression converges for all values of $s \in \lvC \backslash\{1\}$, and (as was known classically) at non-positive integers $s=-k$ one finds 
\begin{equation}
\label{HB}
\zeta(-k,z) = -\frac{B_{k+1}(z)}{k+1},
\end{equation}
where $B_{k}(z)$ is the $k$th {\sc Bernoulli} polynomial. (Again cf. \cite{Cartier}.) Henceforth write $\zeta(-k,z)=:\zeta(-k|z).$

These normalized {\sc Bernoulli} polynomials thus belong to the family of {\sc Hurwitz} zeta functions, and if $\zeta(s)$ denotes the {\sc Riemann} zeta function and $B_{k}:=B_{k}(0)=(-1)^{k}B_{k}(1)$ is the $k$th {\sc Bernoulli} number, 
then the fact that 
\[
\zeta(-k) = (-1)^{k}\frac{B_{k+1}}{k+1}
\]
concords with  $\zeta(s, 1) = \zeta(s)$.

Various generalizations of these polynomials to multiple versions exist, going back to work of {\sc Barnes} in 1899 and much more recently 
{\sc Szenes} in \cite{Szenes}, and {\sc Komori, Matsumoto} and {\sc Tsumura} in \cite{KMT}.  Our approach here is somewhat different however, being determined by solving the system of difference equations
\begin{equation}
\label{B0}
V(-k_{1}, \ldots, -k_{r}|z+1) - V(-k_{1}, \ldots, -k_{r}|z) = 
-z^{k_{1}}V(-k_{2}, \ldots, -k_{r}|z+1)
\end{equation}
where the $k_{j}$ are non-negative, for all possible $r \geq 1.$ Of course, these are the difference equations (\ref{fun}) satisfied by the {\sc Hurwitz}  polyzeta functions, but at parameters $(-k_1, \ldots, -k_r)$ at which the polyzeta functions do not exist, (cf. \cite{Akiyama:etal}).

\begin{thm}
\label{formula}
Up to addition of a function of period 1, the solutions to (\ref{B0}) are given by the recursive formula
\[
\zeta(-k_1,\ldots, -k_r|z)
=
-\frac{1}{k_r+1}\zeta(-k_1, \ldots, -k_{r-2}, -k_{r-1}-k_r-1|z)
\]
\[
-\frac{1}{2}\zeta(-k_1, \ldots, -k_{r-2}, -k_{r-1}-k_r|z)
+
\sum_{q=1}^{k_{r}}(-k_r)_q\frac{B_{q+1}}{(q+1)!}\zeta(-k_1, \ldots, -k_{r-2},-k_{r-1}-k_r +q|z),
\]
writing $\zeta(-k|z) = -\frac{B_{k+1}(z)}{k+1}$ as above.
\end{thm}
\begin{proof}
As is shown in 
\cite{Meschkowski} the (normalized) {\sc Bernoulli} polynomials $\zeta(-n|z)$ give the unique solution to 
\begin{equation}
\label{B1}
u(z+1)-u(z) = -z^{n}
\end{equation}
for any $n \geq 0,$ up to addition of a periodic function - i.e. the solution to (\ref{B0}) when $r=1$ corresponds to $\zeta(-k_1|z)=-\frac{B_{k_{1}+1}(z)}{k_{1}+1}.$ 

Since each $B_{n}(z)$ is polynomial, when $r=2$ the
right side of (\ref{B0}) is also polynomial.  Consequently, by linearity, (\ref{B0}) reduces to a sum of equations of the form of (\ref{B1}), so that some linear combination of usual {\sc Bernoulli} polynomials  gives a solution. Then for $r=3,$ again the right hand side of (\ref{B0}) is polynomial so may be solved by the same method. Proceeding inductively, it is clear that polynomial solutions exist for all possible $r$, which may be explicitly determined: 
 
Consider firstly the case that $r=2,$ and recall the well-known formula
\[
B_{n}(z)= \sum_{j=0}^{n}
\left(
\begin{array}{c}
n \\ j
\end{array}
\right) B_{j}z^{n-j}\;.
\]
Using (\ref{B1}) with $n=k_{2},$ 
(\ref{B0}) then becomes
\[
V(-k_{1}, -k_{2}|z+1)- V(-k_{1},-k_{2}|z) = z^{k_{1}+k_{2}}+\frac{1}{k_{2}+1}\sum_{j=0}^{k_{2}+1}\left(\begin{array}{c}
k_{2}+1 \\ j
\end{array}
\right) B_{j}z^{k_{1}+k_{2}+1-j}
\]
for which a solution is
\[
-\zeta(-k_1-k_{2}|z)-\frac{1}{k_2+1}\sum_{j=0}^{k_{2}+1}\left(\begin{array}{c}
k_{2}+1 
\\ j
\end{array}
\right) B_{j}\cdot\zeta(-k_{1}-k_{2}+j-1|z)
=:\zeta(-k_{1}, -k_{2}|z).
\]
Now write out the $j=0$ and $j=1$ terms of the sum, set $l=j-1,$ and use the {\sc Pochhammer} symbol notation to write 
\[
k_2 (k_2-1) \ldots (k_2 -(l-1)) = (-k_2)_l (-1)^{l}.
\]
One finds
\begin{align*}
&\zeta(-k_1, -k_2|z)
\\
&=
-\frac{1}{2}\zeta(-k_1-k_2|z)-
\frac{1}{k_{2}+1}\zeta(-k_1-k_2-1|z)
\\
&
+\sum_{l=1}^{k_{2}}(-k_2)_{l}(-1)^{l+1} \cdot \frac{B_{l+1}}{(l+1)!}\cdot \zeta(-k_{1}-k_{2}+l|z) 
\end{align*}
The $(-1)^{l+1}$ factor of the last line may be omitted since $B_{l+1}=0$ for all even $l \geq 2.$

Continue by induction, supposing that it is known that a solution to (\ref{B0}) has the requisite form for all positive integers $r$ less than or equal to some $N-1$. Then by this inductive hypothesis, 
\begin{align*}
&
\zeta(-k_1, \ldots, -k_N|z+1)
-
\zeta(-k_1, \ldots, -k_N|z)
\\
&=
-z^{k_1}\zeta(-k_2, \ldots, -k_N|z+1)
\\
&=
-z^{k_1}\left[
-\frac{1}{k_N+1}\zeta(-k_2, \ldots, -k_{r-2}, -k_{N-1}-k_N-1|z+1)\right.
\\
&-\frac{1}{2}\zeta(-k_2, \ldots, -k_{N-2}, -k_{N-1}-k_N|z+1)
\\
&
+\left.
\sum_{q=1}^{k_{N}}(-k_N)_q\frac{B_{q+1}}{(q+1)!}\zeta(-k_2, \ldots, -k_{N-2},-k_{N-1}-k_N +q|z+1)\right].
\end{align*}
The difference equation obtained here may be regarded as a sum of difference equations of the form of 
\[
v(z+1) - v(z) = -Az^{k_1}w(z+1)
\]
where in each case, $w(z)$ is some normalized multiple {\sc Bernoulli} polynomial of depth $N-2$ and $A$ is some rational number. 
Each of these may be solved and the sum of these solutions gives the desired formula in the case that $r=N.$
\end{proof}

\begin{cor}
\label{3.2}
\[
\zeta(-k_1, \ldots, -k_r|1)
=
\lim_{s_1 \rightarrow -k_1}
\cdots
\lim_{s_r\rightarrow -k_r}
\zeta(s_1, \ldots, s_r),
\]
where $\zeta(s_1, \ldots, s_r)$ denotes the meromorphic continuation of the series
\[
\sum_{0<n_1\ldots<n_r}\frac{1}{n_1^{s_{1}}\ldots n_r^{s_r}}.
\]
\end{cor}
This is immediate from the Theorem combined with the expression for the limit given in \cite{Akiyama:etal}.

\subsection{Elementary facts pertaining to values of NMBPs}

\begin{prop}{
\[
\zeta(0)\zeta(-2n-1)
= \zeta(-2n+k-1,-k|1)
\]
for $n\geq 1$ and $k$ with $0\leq k \leq 2n+1.$}
\end{prop}
\begin{proof}
\[
\zeta(-2n+k-1,-k|z) = 
\frac{B_{2n+2}(z)}{2n+2}+
\frac{1}{k+1}\sum_{j=0}^{k+1}\left(
\begin{array}{c}
k+1 \\ j
\end{array}
\right) B_{j}\cdot\frac{B_{2n-j+3}(z)}{2n-j+3}
\]
so that
\[
\zeta(-2n+k-1,-k|1) = \frac{(-1)^{2n+2}B_{2n+2}}{2n+2}+
\frac{1}{k+1}\sum_{j=0}^{k+1}\left(\begin{array}{c}
k+1 \\ j
\end{array}
\right) B_{j}\cdot\frac{(-1)^{2n-j+3}B_{2n-j+3}}{2n-j+3}
\]

In the sum that appears on the right side, since $j$ and $2n-j+3$ have opposite parity, the only non-zero term is that for which $j=1$. (Notice that $2n-j+3 \geq 2.$)  

Hence, 
\[
\zeta(-2n+k-1,-k|1) = \frac{B_{2n+2}}{2n+2} - \frac{1}{2}\frac{B_{2n+2}}{2n+2} = -\frac{1}{2}\left(-\frac{B_{2n+2}}{2n+2}\right) = \zeta(0)\zeta(-2n-1).
\]
\end{proof}

\begin{prop}
For any integer $n>0,$
\[
\zeta(-n,-n-1|1) = \zeta(-n-1,-n|1).
\]
\end{prop}
\begin{proof}
Using
\[
\frac{1}{n+2}\left(\begin{array}{c}
n+2 \\
j\end{array}
\right)-\frac{1}{n+1}\left(\begin{array}{c}
n+1\\j\end{array}\right)
=
\left(\begin{array}{c}
n\\
j
\end{array}\right)
\cdot\frac{j-1}{(n+2-j)(n+1-j)}
\]
one computes
\begin{align*}
&
\zeta(-n,-n-1|z) -\zeta(-n-1,-n|z)
\\
&=
\sum_{j=1}^{n+1}\left(\begin{array}{c}
n\\j\end{array}\right)
\frac{j-1}{(n+2-j)(n+1-j)}
B_j\cdot\frac{B_{2n+3-j}(z)}{2n+3-j}
+
\frac{B_{n+2}}{n+2}\frac{B_{n+1}(z)}{n+1}.
\end{align*}
Now evaluate at $z=1$ and observe that the parity of $j$ and $2n+3-j$ differs, while $2n+3-j \geq 2$ since $n > 0.$ Also, the parity of $n+2$ and $n+1$ differs. Thus we can use the same argument as in the proof of the previous proposition, to conclude that 
\[
\zeta(-n,-n-1|1) - \zeta(-n-1,-n|1) = 0.
\]
\end{proof}

\begin{prop}
\[
\zeta(-2n_1-1,0,-2n_2-1|1)
=
-\zeta(-2n_1-1,-2n_2-1|1)
\]
for any $n_1, n_2\geq 0.$
\end{prop}
\begin{proof}
Once the polynomials $\zeta(-2n_1-1,0,-2n_2-1|z)$ and $\zeta(-2n_1-1,-2n_2-1|z)$ have been determined explicitly, evaluation at $z=1$ gives an expression which may be readily simplified by considering the parity of the indices $j,k$ in the products $B_j \cdot B_k$ which arise, exploiting the fact that the only odd $j$ for which $B_j$ is non-zero, is $j=1$.
One obtains
\begin{align*}
&
\zeta(-2n_1-1, 0,-2n_2-1|1)
\\
&=
\frac{B_{2n_1+2n_2+3}}{2n_1+2n_2+3}
-\frac{1}{2n_2+2}
\sum_{m=0}^{2n_2+2}\left(
\begin{array}{c}
2n_2+2
\\
m
\end{array}
\right)B_m
\frac{B_{2n_1+2n_2+4-m}}{2n_1+2n_2+4-m}.
\\
&=
-\zeta(-2n_1-1, -2n_2-1|1).
\end{align*}
\end{proof}

It is likely that also
\[
\zeta(-2n+k, 0, -k|z) = \zeta(-2n+k, -k|z)\]
for any $k, n$ with $0 <k \leq 2n.$

It is convenient to introduce the following notation:
If $f(z)=\sum_{n=0}^{k}a_nz^n,$ write $f(B)(z)$ for $\sum_{n=0}^{k}a_n\frac{B_{n+1}(z)}{n+1}.$
Also, for $\zeta(n_1, \ldots, n_r|z)$ with $n_1=\ldots =n_r=n$, write
$\zeta_r(n|z).$

\begin{prop} 
\[
\zeta_r(0|1)= (-1)^{r}\frac{1}{r+1}
\] and
\[
\zeta_r(0|0) = (-1)^{r+1}\frac{1}{r(r+1)}.
\]
\end{prop}
The first assertion is proven (by different means) for the limit expression of Corollary \ref{3.2} in \cite{Akiyama}.
\begin{proof}
First observe that by applying the difference equation satisfied by the multiple {\sc Bernoulli} polynomials recursively, 
\[
\zeta_r(0|z+1) -\zeta_r(0|z)= \sum_{n=1}^{r-1}(-1)^n \zeta_{r-n}(0|z).
\]
Hence \[
\zeta_r(0|z)  = \sum_{n=1}^{r-1}(-1)^n \zeta_{r-n}(0|B)(z).
\]
But then taking $t=r$ and $t-1=r$ in this equation successively and adding, 
\begin{equation}
\label{e:B}
\zeta_t(0|z)+\zeta_{t-1}(0|z)
=
-\zeta_{t-1}(0|B)(z).
\end{equation}

Now since $\zeta(0|z) = -z+\frac{1}{2},$  
\[
\zeta(0,0|z) =\frac{z^2}{2}-\frac{1}{6}\]
and
\[
\zeta(0,0,0|z) = -\frac{z^3}{6}-\frac{z^2}{4}+\frac{z}{12}+\frac{1}{12},
\]
the assertion is seen to hold for $r=1,2,3.$ Now when the statement holds for both $r$ and $r+1,$ then by (\ref{e:B}), 
\begin{equation}
\label{e:B0}
-\zeta_r(0|B)(0)=\frac{(-1)^{r+1}2}{(r+2)(r+1)r},
\end{equation}
which is true in the case of $r=1,2$ by the above.

We proceed by induction, supposing that (\ref{e:B0}) is known to hold for all $r$ less than or equal to some $K$. Also suppose that the assertion of the theorem is known to hold for such $r$.

Now (\ref{e:B}) holds for all $t$ - in particular for $t=K+1.$ But then
\begin{align*}
\zeta_{K+1}(0|0)
&=-\zeta_{K}(0|0) - \zeta_{K}(0|B)(0)
\\
&=
-\frac{(-1)^{K+1}}{K(K+1)} +\frac{(-1)^{K+1}2}{(K+2)(K+1)K}
\\
&=
\frac{(-1)^{K+2}}{(K+2)(K+1)},
\end{align*}
proving the assertion regarding evaluation of $\zeta_r(0|z)$ at $z=0,$ by induction.

 For the other statement, recall the difference equation:
 \[
 \zeta_{K+1}(0|z+1)-\zeta_{K+1}(0|z)  = -\zeta_{K}(0|z+1).
 \]
Evaluating at $z=0$ gives
\begin{align*}
\zeta_{K+1}(0|1)
& = 
\zeta_{K+1}(0|0)-\zeta_{K}(0|1)
\\
&=
\frac{(-1)^{K+2}}{(K+2)(K+1)}
- \frac{(-1)^{K}}{K+1}\;\;\;\text{ as above, and by induction hypothesis}
\\
&=
\frac{(-1)^{K+1}}{K+2}.
\end{align*}
\end{proof}

\begin{remark}
One can easily deduce from the above proof that also
\[
-\zeta_r(0|B)(0)=\zeta_{r+1}(0|-1) = \frac{(-1)^{r+1}2}{(r+2)(r+1)r}
\]
for $r \geq 1.$
\end{remark}

\begin{prop}
If $(-n_1, \ldots, -n_r)$ is a tuple of non-positive integers, and $-n_1<0,$ then 
\[
\zeta(-n_1, \ldots, -n_r|1) = \zeta(-n_1, \ldots, -n_r|0).
\]
\end{prop}
\begin{proof}
Because $B_{2n+1}=0$ unless $n=0,$ the assertion is known from the equality $B_n(0) = (-1)^n B_n(1)$ in the case that $r=1.$

The general statement now follows from the formula in Theorem \ref{formula}, by an easy induction. 
\end{proof}

\subsection{The algebraic {\sc Knizhnik-Zamolodchikov} equation for MBPs}
Let $\{y_{-n}\}_{n=0}^{\infty}$ denote an infinite family of formal non-commuting variables.  To any monomial in the $y_{j}$ we associate a normalized multiple {\sc Bernoulli} polynomial (NMBP subsequently) in the obvious way, namely
\[
y_{i_{1}}\cdots y_{i_{r}} \mapsto \zeta(y_{i_{1}}\cdots y_{i_{r}}|z):=\zeta(-{i_{1}},\ldots, -{i_{r}}|z).
\]

Now take $\{Y_{-n}\}_{n=0}^{\infty}$ to denote some other family of formal non-commuting variables.
Then form the generating function of NMBPs:
\[
H_B(z):=\sum_{w \in Y^{*}}\zeta(w|z)W
\]
where the sum is taken over all words $w$ in the $y_{j}$ and $W$ denotes the corresponding word in the $Y_{j}$ (so that $y_{i_{1}}\cdots y_{i_{r}}$ is associated to $Y_{i_{1}}\cdots Y_{i_{r}}$).

{\thm{
\[
H_B(z+1) - H_B(z) = -\sum_{k=0}^{\infty}Y_{-k}z^{k}H_B(z+1).
\]
}}
{\bf Proof:}
Fix $k \geq 0.$ Then
\[
\sum_{w\in Y^{*}}\zeta(y_{-k}w|z+1)Y_{-k}W-\sum_{w\in Y^{*}}\zeta(y_{-k}w|z)Y_{-k}W
=-z^{k}Y_{-k}\sum_{w\in Y^{*}}\zeta(w|z+1)W.
\]
Adding over all $k \geq 0$ gives the result of the theorem.
$\hfill \Box$

This proof is virtually identical to that of Theorem \ref{t:AKZ} for the corresponding fact for {\sc Hurwitz} polyzeta functions.  
Consequently, the next result follows trivially:
\begin{cor}
\[
D_{-}+\sum_{k=0}^{\infty}Y_{-k}z^{k}
\]
is a difference connection on $\lvP^1.$
\end{cor}

Now it is very tempting to try to form a generating function encompassing both classes of functions discussed so far. However, this would necessarily include polyzeta functions of the form of $\zeta(k_1, \ldots , k_r|z)$ for arbitrary tuples of integers $k_j$, not all of which can be determined by difference equation methods. In fact, by a similar recursive procedure to that employed above, it is clear that one can determine functions of the form of $\zeta(-k_1, \ldots, -k_v,k_{v+1}, \ldots, k_r|z)$ where $k_j \geq 0$ for all $j\leq v$, and $k_{j}>0$ for $v<j\leq r,$ but for more general tuples of both positive and non-positive integers, the difference equation method of \cite{Meschkowski} fails in that non-convergent series would arise as the supposed solutions.

\subsection{Algebraic independence of {\sc Hurwitz} polyzeta functions}

\begin{thm}
The {\sc Hurwitz} polyzeta functions are  algebraically independent over $\lvC$ - i.e for any $N \geq 1,$ if $P(T_{1}, \ldots, T_{N})\in \lvC[T_{1}, \ldots, T_{N}]$ is a polynomial which vanishes (uniformly in $z$) at some $N$-tuple of any {\sc Hurwitz} polyzeta functions of the variable $z$, then $P$ is identically zero.
\end{thm}
{\bf Proof:}
Firstly, the linear independence of the {\sc Hurwitz} polyzeta functions is an easy consequence of the difference equations (\ref{fun}):
Consider an arbitrary linear combination of a 
constant function and {\sc 
Hurwitz} polyzeta functions of depths two and one, which sums to zero:
\[
\lambda +\sum_{k_{1};k_{2}}\lambda_{k_{1}k_{2}}\zeta(k_{1}, k_{2}; 
z)+\sum_{k}\lambda_{k}\zeta(k;z)  = 0.
\]
This equation is also valid replacing $z$ by $z+1,$ so that from the 
functional equations (\ref{fun}), we find that also
\[
\sum_{k_{1};k_{2}}\lambda_{k_{1}k_{2}}\left(
-\frac{1}{z^{k_{1}}} \sum_{m=1}^{\infty}\frac{1}{(z+m)^{k_{2}}}
\right)
-\sum_{k}\lambda_{k}\frac{1}{z^k}  =0.
\]
But now it is clear that the coefficients $\lambda_{k_{1}k_{2}}$ and 
$\lambda_{k}$ must all be zero, since $\frac{1}{z^{k}}$ has a pole of 
order exactly $k$ at $z=0,$ and 
$\frac{1}{z^{k_{1}}} \sum_{m=1}^{\infty}\frac{1}{(z+m)^{k_{2}}}$ has a 
pole of order $k_{1}$ at $z=0$ and a pole of order $k_{2}$ at $z=-1.$
Then trivially, also $\lambda = 0.$

Now suppose that one has shown the linear independence of a $\lvC$-linear 
combination of a constant function and {\sc Hurwitz} polyzeta functions of 
depths $1, \ldots, m-1$. For constants $\lambda_{l^{j}_{1}\ldots 
l^{j}_{j}} \in \lvC,$ consider now
\[
0=
\lambda
+\sum_{l_{1}^{m}, \ldots, l_{m}^{m}} \lambda_{l_{1}^{m}\ldots l_{m}^{m}}
\zeta(l_{1}^{m}, \ldots, l_{m}^{m};z)
\]
\[
+\sum_{l_{1}^{m-1},\ldots,l_{m-1}^{m-1}} 
\lambda_{l_{1}^{m-1}\ldots l_{m-1}^{m-1}}
\zeta(l_{1}^{m-1}, \ldots, l_{m-1}^{m-1};z) 
+ \ldots
+\sum_{l_{1}^{1}} \lambda_{l_{1}^{1}}\zeta(l_{1}^{1};z).
\]

Because this relation must also hold for $z+1,$ using (\ref{fun}) 
and subtracting the two equations from each other, we find that
\[
0=
\sum_{l_{1}^{m}, 
\ldots, l_{m}^{m}} 
\lambda_{l_{1}^{m}\ldots l_{m}^{m}}
\frac{1}{z^{l_{1}^{m}}}
\zeta(l_{2}^{m}, \ldots, 
l_{m}^{m};z+1)
\]
\[
+\sum_{l_{1}^{m-1},\ldots,l_{m-1}^{m-1}} 
\lambda_{l_{1}^{m-1}\ldots l_{m-1}^{m-1}}
\frac{1}{z^{l_{1}^{m-1}}}
\zeta(l_{2}^{m-1}, \ldots, l_{m-1}^{m-1};z+1)
+ 
\ldots+\sum_{l_{1}^{1}} \lambda_{l_{1}^{1}}
\frac{1}{z^{l_{1}^{1}}}.
\]
Now let $$L = \prod_{j=1}^{m}\prod_{l_{1}^{j}} l_{1}^{j}$$ and multiply 
the equation by 
$z^{L}.$ Then each function 
$z^{L-l_{1}^{j}}\zeta(l_{2}^{j}, \ldots, l_{j}^{j};z+1)$ has a zero of order 
exactly $L-l_{1}^{j}$ at $z=0$. Functions with zeros of distinct orders 
at a point of $\lvC$ are linearly independent over $\lvC,$ so the equation 
breaks up into distinct parts where the $l_{1}^{j}$-values agree. In this 
way a system of equations of the form:
\[
\sum_{l_{1}^{m}, \ldots, l_{m}^{m}} 
\lambda_{l_{1}^{m}\ldots l_{m}^{m}}
{z^{L-l}}\zeta(l_{2}^{m}, \ldots, l_{m}^{m};z+1)
+ \ldots+\sum_{l_{1}^{1}} \lambda_{l_{1}^{1}}{z^{L-l}} = 0
\]
results. (Here $l$ denotes the common value of the indices $l_{1}^{m}, 
\ldots, 
l_{1}^{1}$.)
But these equations can be divided by $z^{L-l}$ and then the inductive 
hypothesis guarantees that all of the coefficients are zero. This 
concludes the linear independence proof.

Thanks to the stuffle product, the functions are also algebraically independent, since any supposed algebraic relation among such functions could be decomposed by means of the stuffle product into a linear expression. In such a linear  expression, the coefficients are sums of the coefficients of the original algebraic expression. However,
at least one of the resulting polyzeta functions only arises from a single term of the original expression, the coefficient of which must therefore be zero. This eliminates certain terms from the linear expression, and once again, at least one of the remaining polyzeta functions comes out of a unique term of the algebraic expression, so that this latter term again has zero coefficient by the linear independence. Continuing inductively in this way, the theorem is proven.
$\hfill \Box$

\bibliographystyle{alpha}	
\bibliography{bibliog}

{Sheldon T Joyner, Mathematics Department, The University of Western Ontario, Middlesex College, London, Ontario N6A 5B7, Canada}
\newline
{\em  email:} {\tt{sjoyner at uwo dot ca}}

\end{document}